\documentclass[11pt,a4paper]{article}
\usepackage{amsmath}
\usepackage{amsthm}
\usepackage{amsfonts}
\usepackage{multirow}
\usepackage[table,xcdraw]{xcolor}
\usepackage{subcaption}
\usepackage{float}
\usepackage{algorithm}
\usepackage{algpseudocode}
\usepackage{hyperref}
\usepackage{tabularx}
\usepackage{amssymb}
\usepackage{graphicx}
\usepackage{tcolorbox}
\usepackage{booktabs}
\usepackage[table,xcdraw]{xcolor}
\usepackage{multirow}
\usepackage{rotating,tabularx}

\usepackage[left=2.5cm,right=2.5cm,top=2.5cm,bottom=2.5cm]{geometry}
\newcolumntype{L}[1]{>{\raggedright\let\newline\\\arraybackslash\hspace{0pt}}m{#1}}
\newcolumntype{C}[1]{>{\centering\let\newline\\\arraybackslash\hspace{0pt}}m{#1}}
\newcolumntype{R}[1]{>{\raggedleft\let\newline\\\arraybackslash\hspace{0pt}}m{#1}}

\newtheorem{Theorem}{Theorem}[section]
\newtheorem{Proposition}[Theorem]{Proposition}
\newtheorem{Remark}[Theorem]{Remark}
\newtheorem{Lemma}[Theorem]{Lemma}
\newtheorem{Corollary}[Theorem]{Corollary}

\newtheorem{Example}[Theorem]{Example}

\usepackage{hyperref}
\hypersetup{colorlinks=true,urlcolor=blue}
\expandafter\let\expandafter\oldproof\csname\string\proof\endcsname
\let\oldendproof\endproof
\renewenvironment{proof}[1][\proofname]{
\oldproof[\ttfamily\scshape \bf #1.]
}{\oldendproof}

\def\dom{{\rm dom}\,}

\def\epi{{\rm epi\,}}

\def\ox{\overline{x}}
\def\oy{\overline{y}}
\def\oz{\overline{z}}

\def\Hat{\widehat}

\def\Bar{\overline}
\def\ra{\rangle}
\def\la{\langle}

\def\epsilon{\varepsilon}
\def\ox{\bar{x}}
\def\oy{\bar{y}}
\def\oz{\bar{z}}

\def\ov{\bar{v}}

\def\gph{\mbox{\rm gph}\,}
\def\epi{\mbox{\rm epi}\,}

\def\dom{\mbox{\rm dom}\,}

\def\dn{\downarrow}
\def\O{\Omega}
\def\ph{\varphi}

\def\st{\stackrel}
\def\oR{\Bar{\R}}
\def\lm{\lambda}

\def\al{\alpha}

\def \N{{\rm I\!N}}
\def \R{{\rm I\!R}}

\def\Limsup{\mathop{{\rm Lim}\,{\rm sup}}}

\def\Limsup{\mathop{{\rm Lim}\,{\rm sup}}}

\numberwithin{equation}{section}

\title{\bf  New Characterizations of Nonsmooth Convex Functions via Generalized Derivatives}
\author{Vo Thanh Phat\footnote{Department of Mathematics and Statistics,  University of North Dakota, Grand Forks, North Dakota, USA. E-mail: thanh.vo.1@und.edu.}}
\begin{document}  
\maketitle

\noindent
{\small{\bf Abstract}.  This paper studies the convexity properties of nonsmooth extended-real-valued weakly convex functions, a class of functions that is central to modern optimization and its applications. We establish new characterizations of convexity using second-order generalized derivative tools, including subgradient graphical derivatives, second subderivatives, and second-order subdifferentials. These tools allow us to derive necessary and sufficient conditions for convexity in the nonsmooth framework.  \\[1ex]
{\bf Key words}. Convexity, weakly convex functions,  second-order subdifferentials,  graphical derivatives, second subderivatives  \\[1ex]
{\bf Mathematics Subject Classification (2010).} 49J52, 49J53, 90C30, 26B25 }\vspace*{-0.1in}

\section{Introduction}
Convexity plays a fundamental and important role in optimization theory and numerical methods. As is well known, the  {second-order} derivative test in classical calculus provides a powerful tool for identifying convexity.  {More specifically, a twice continuously differentiable function $\varphi:\Omega\to\R$, where $\Omega\subset\R^n$ is an open convex set, is convex if and only if its Hessian matrix $\nabla^2\varphi(x)$ is positive semidefinite for all $x\in\Omega$.} However, in many practical models, the first or second-order derivatives of a function may not exist, which makes it impossible to apply this criterion to characterize convexity, especially for nonsmooth extended-real-valued functions.

\medskip To address the nonsmooth case, many characterizations utilizing second-order generalized derivatives for convex functions have been investigated. In \cite{ChieuChuongYaoYen}, the authors employ the second-order subdifferential introduced by Mordukhovich in \cite{m92} to establish a necessary condition for nonsmooth convex functions, as well as sufficient conditions for convexity in several special cases, including $\mathcal{C}^{1,1}$ functions, $\mathcal{C}^{1}$ univariate functions, piecewise linear functions, and certain classes of piecewise $\mathcal{C}^2$ functions. The development of {characterizations for convexity of piecewise $\mathcal{C}^2$ functions} was then completed by Chieu and Yao in \cite{chieuyao11}. In \cite{chieuhuy}, Chieu and Huy provided a characterization of $\mathcal{C}^1$-smooth functions via second-order subdifferentials and showed that this result cannot be extended to general locally Lipschitz functions. The characterization of convexity for lower-$\mathcal{C}^2$ functions was later investigated in \cite{clmn16}. More recently, Nadi and Zafarani proposed another sufficient condition for the convexity of lower-$\mathcal{C}^1$ functions in \cite{nz23}.

\medskip To the best of our knowledge, although many conditions for characterizing the convexity of nonsmooth functions have been discussed above, the characterization of convexity for nonsmooth  {extended-real-valued} functions via  {second-order} subdifferentials has not yet been investigated. Moreover, all previously mentioned conditions require the functions to be locally Lipschitz continuous and real-valued. In addition, while  {second-order} subdifferentials have proven to be powerful tools in variational analysis, both in theory and in numerical algorithms as shown in \cite{Mordukhovich06,Mor18,Mor24}, other generalized derivatives that are also important in theoretical and numerical optimization, including subgradient graphical derivatives and second subderivatives as presented in \cite{chnt,ChieuNghia,gtdquad,mms,BorisEbrahim}, have not yet been explored for characterizing the convexity of nonsmooth functions.

\medskip In this paper, we investigate new characterizations of convexity for extended-real-valued weakly convex functions, a class that plays a fundamental role in optimization and related areas. Our analysis is conducted in the nonsmooth setting and relies on three central second-order generalized derivatives discussed above. Weak convexity has been extensively studied in variational analysis and optimization due to its rich structure and broad applicability; see, for example, \cite{dm05,davisdrus22,Rockafellar98,vial83}. In particular, this class supports the convergence analysis of a variety of numerical methods, including proximal-type and first-order algorithms for nonsmooth optimization (cf. \cite{davisdrus22,kmptjogo25}). Weak convexity also plays an important role in the theory of differential equations, where it typically appears under the name of semiconvexity. In particular, semiconvex functions arise naturally in the analysis of Hamilton-Jacobi equations, viscosity solutions, and optimal control problems, where they provide essential regularity properties such as almost-everywhere second-order differentiability and stability under infimal convolutions; see, for example, \cite{cs04,cil92}. These properties are crucial for the well-posedness of solutions, the derivation of comparison principles, and the qualitative analysis of nonsmooth dynamics. This perspective further underscores the broad relevance of weakly convex functions across optimization, variational analysis, and partial differential equations.

\medskip The rest of the paper is organized as follows. Section \ref{sec:prelim} presents the necessary definitions and results from variational analysis.  {In Section \ref{sec:graphical}, we establish necessary and sufficient conditions for nonsmooth convex functions using generalized second-order derivative constructions. Section \ref{sec:ex} presents illustrative examples and further discussions of our results, particularly for functions commonly encountered in practical models.} Finally, Section \ref{sec:conclusion} offers concluding remarks and outlines possible directions for future research.
 
\section{Preliminaries from Variational Analysis}\label{sec:prelim}
This section recalls some notions and notation frequently used in this paper. All our considerations are given in the space $\R^n$ with the \textit{Euclidean norm} $\|\cdot\|$ and \textit{scalar/inner product} $\langle \cdot,\cdot\rangle$ defined by
$$
\|x\|:= \left(\sum_{i=1}^n x_i^2 \right)^{\frac{1}{2}} \; \text{and } \langle x,y\rangle =\sum_{i=1}^n x_i y_i \quad \text{where } x:=(x_1,\ldots, x_n), \; y:=(y_1,\ldots,y_n).
$$
Throughout the paper, $\overline{\R} := \R \cup \{\infty\}$ denotes the extended real line, and $\N := \{1,2,\ldots\}$ denotes the set of natural numbers.

\medskip 
We provide the necessary background on variational analysis and generalized differentiation, drawing on the foundational works of Mordukhovich \cite{Mordukhovich06,Mor18,Mor24} and Rockafellar and Wets \cite{Rockafellar98}. 
Given a set $\Omega\subset\R^s$ with $\oz\in\O$, the (Bouligand-Severi) {\em tangent/contingent cone} to $\O$ at $\oz$ is
\begin{equation}\label{tan}
T_\O(\oz):=\big\{w\in\R^s\;\big|\;\exists\,t_k\dn 0,\;w_k\to
w\;\mbox{ as }\;k\to\infty\;\mbox{ with }\;\oz+t_k w_k\in\O\big\}.
\end{equation} 
The (Fr\'echet) {\em regular normal cone} to $\Omega$
at $\bar{z}\in\Omega$ is defined by \begin{equation}\label{rnc}
\widehat{N}_\Omega(\bar{z}):=\Big\{v\in\R^s\;\Big|\;\limsup_{z\overset{\Omega}{\rightarrow}\bar{z}}\frac{\langle
v, z-\bar{z}\rangle}{\|z-\bar{z}\|}\le 0\Big\}, \end{equation} 
where
the symbol $z\overset{\Omega}{\rightarrow}\bar{z}$ indicates that $z\to\bar{z}$ with $z\in\Omega$. It can be equivalently described via a duality correspondence with \eqref{tan} by
\begin{equation}\label{dua} \Hat
N_\O(\oz)=T^*_\O(\oz):=\big\{v\in\R^s\;\big|\;\la v,w\ra\le
0\;\mbox{ for all }\;w\in T_\O(\oz)\big\}. \end{equation} The
(Mordukhovich) {\em limiting normal cones} to $\Omega$ at
$\bar{z}\in\Omega$ is defined by \begin{equation}\label{lnc}
N_\Omega(\bar{z}):=\big\{v\in\R^s\;\big|\;\exists\,z_k\st{\O}{\to}\bar{z},\;v_k\to
v\;\text{ as }\;k\to\infty\;\text{ with
}\;v_k\in\widehat{N}_\Omega(z_k)\big\}. \end{equation} 
It is important to observe that the regular normal cone defined in \eqref{rnc} is always convex, whereas the limiting normal cone in \eqref{lnc} can be nonconvex, as illustrated by the graph of $|\cdot|$ at $\bar{z}=(0,0)$. Consequently, it cannot be derived through a duality relation of the form \eqref{dua} from any tangential approximation of $\Omega$ at $\bar{z}$. Despite this, the limiting normal cone \eqref{lnc}, along with the associated  subdifferential constructions for  functions, {enjoys} full calculus rules that are grounded in the variational and extremal principles of variational analysis.

\medskip Let $F$ be a set-valued mapping  between Euclidean spaces $\R^n$ and $\R^m$. As usual, the effective domain and the graph of $F$ are given, respectively, by
$$
\text{dom}F:=\{x\in \R^n|F(x)\ne\emptyset\}\quad\text{and}\quad\text{gph}F=\{(x,y)\in\R^n\times\R^m|y\in F(x)\}.
$$
The  (Painlev\'e-Kuratowski) \textit{outer limit} of $F$ as $x\rightarrow\bar{x}$ is given by
\begin{equation}\label{outer}
\underset{x\rightarrow\bar{x}}{\Limsup}\;F(x):
=\big\{ {y\in \R ^m}\;\big|\,\exists \mbox{ sequences } x_k\to \bar x,\ y_k\rightarrow y\;\mbox {with }\;y_k\in F(x_k),\;k\in\N\big
\}.
\end{equation}
The {\em graphical derivative} of $F$ at
$(\ox,\oy)\in\gph F$ is defined via \eqref{tan} by
\begin{equation}\label{gra-der}
DF(\ox,\oy)(u):=\big\{v\in\R^m\;\big|\;(u,v)\in T_{{\rm
		gph}\,F}(\ox,\oy)\big\},\quad u\in\R^n. 
\end{equation} 
This concept was introduced in the early 1980s by Jean-Pierre Aubin, who called it the contingent derivative. Here we follow the references \cite{Donchev09,Rockafellar98}
in using the terminology ``the graphical derivative". If $(\ox,\oy)\notin \gph F$, one puts $DF(\ox,\oy)(w) = \emptyset$ for all $w$ by convention.  In terms of the outer limit given in \eqref{outer} and the difference quotient mapping 
$$
\Delta_t F(x,y)(w)= \frac{1}{t}[F(x+tw)-y] \quad \text{for } y \in F(x), \; t>0,
$$
we can represent the graphical derivative in \eqref{gra-der} as follows:
\begin{equation}\label{gra-derII}
DF(\ox,\oy)(u)= \Limsup_{t \downarrow 0, w \to u} \Delta_t F(\ox,\oy)(w), \quad u \in \R^n. 
\end{equation}
By the definition of graphical derivative, we have the following relationship
\begin{equation}\label{relainverse}
v \in DF(\ox,\oy)(u) \iff u \in DF^{-1}(\oy,\ox)(v) \quad \text{for all } u \in \R^n, v\in\R^m. 
\end{equation} 
In the case where $F(\bar{x})$ is the singleton $\{\bar{y}\}$, we omit $\oy$ in
the notation of \eqref{gra-der}.   Note that if the mapping
$F\colon\R^n\to\R^m$ is a single-valued differentiable mapping  at $\ox$, then $ {DF(\bar{x})(u)=\{\nabla F(\bar{x})u\}}$ for all $u\in\R^n$, where $\nabla F(\ox)$ is the Jacobian matrix of $F$ at $\ox$. 
Using the normal cones \eqref{rnc} and \eqref{lnc} to the graph of $F$, we define the {\em regular coderivative} and the {\em limiting coderivative} of $F$ at $(\ox,\oy)\in\gph F$, respectively, by
\begin{equation}\label{reg-cod} 
\Hat D^*F(\ox,\oy)(v):=\big\{u\in\R^n\;\big|\;(u,-v)\in\Hat N_{{\rm
gph}\,F}(\ox,\oy)\big\},\quad v\in\R^m, 
\end{equation}
\begin{equation}\label{lim-cod}
D^*F(\ox,\oy)(v):=\big\{u\in\R^n\;\big|\;(u,-v)\in N_{{\rm
gph}\,F}(\ox,\oy)\big\},\quad v\in\R^m. 
\end{equation} 
 {If $F$ is differentiable at $\ox$, then $\widehat{D}^*F(\ox)(v)=\{\nabla F(\ox)^*v\}$ for all $v\in\R^m$. Moreover, if $F$ is of class $\mathcal{C}^1$ around $\ox$, then $D^*F(\ox)(v)=\{\nabla F(\ox)^*v\}$ for all $v\in\R^m$, where $\nabla F(\ox)^*$ denotes the transpose of the Jacobian matrix of $F$ at $\ox$.}

\medskip 
Let us consider an extended-real-valued function $\varphi:\R^n\rightarrow\overline{\R}:=(-\infty,\infty]$, and its
{\em effective domain} and {\em epigraph} defined by, respectively,
\begin{equation*}
\dom\ph:=\big\{x\in\R^n\;\big|\;\ph(x)<\infty\big\}\;\mbox{ and
}\;\epi\ph:=\big\{(x,\al)\in\R^{n+1}\;\big|\;\al\ge\ph(x)\big\}.
\end{equation*}  
We always assume that $\varphi$  is proper and lower semicontinuous (l.s.c.). The  \textit{basic/limiting subdifferential} of $\varphi$ at $\bar{x}\in\text{dom}\varphi$  is defined  by
\begin{equation}\label{lim-sub}
\partial\varphi(\ox):=\big\{v\in\R^n\;\big|\;(v,-1)\in
N_{{\rm \text{\rm epi}}\,\varphi}\big(\bar{x},\varphi(\bar{x})\big)\big\}.
\end{equation} 
Recall that the basic subdifferential
$\partial\varphi(\ox)$ reduces to the gradient $\{\nabla\ph(\ox)\}$
if $\ph$ is ${\cal C}^1$-smooth around $\ox$ (or merely strictly
differentiable at this point), and that $\partial\ph(\ox)$ is the
subdifferential of convex analysis if $\ph$ is convex.  {To construct second-order generalized derivatives, one natural approach follows the classical principle that a second-order derivative is obtained by differentiating a first-order derivative. In this spirit, by applying the graphical derivative \eqref{gra-der} to the subgradient mapping \eqref{lim-sub}, we obtain the \textit{subgradient graphical derivative} $D\partial\varphi$, which has been investigated in \cite{chnt,ChieuNghia}. Another second-order construction arising from this viewpoint is the second-order subdifferential. Specifically, by applying the coderivative constructions \eqref{reg-cod} and \eqref{lim-cod} to the limiting subdifferential mapping \eqref{lim-sub}, we arrive at the second-order subdifferentials of $\varphi\colon \R^n\to\oR$ defined by}
\begin{equation}\label{seccombine} 
\breve{\partial}^2\varphi(\bar{x},\ov)(u):=(\widehat{D}^*\partial\varphi)(\bar{x},\ov)(u)\;\mbox{ and }\;{\partial}^2\varphi(\bar{x},\ov)(u):=({D}^* {\partial}\varphi)(\bar{x},\ov)(u),\quad u\in \R^n,
\end{equation} 
known as the {\em combined second-order subdifferential} and the {\em limiting second-order subdifferential} of $\ph$ at $\bar{x}$ relative to $\ov\in\partial\ph(\ox)$, respectively. The concept of second-order subdifferentials was introduced by Mordukhovich in \cite{m92} and has since been extensively developed to establish calculus rules for broad classes of extended-real-valued functions, derive optimality conditions, and characterize convex and generalized convex functions; see, e.g., \cite{ChieuChuongYaoYen,kp18,kp20,KKMP23-2,KKMP23-3,kmp22convex,BorisOutrata,Poli,roc24} and the references therein. The application of second-order subdifferentials to generalized Newton methods and their use in practical areas such as machine learning, statistics, and image processing have also been investigated in \cite{BorisKhanhPhat,kmp24cod,kmptjogo,kmptmp,BorisEbrahim}.  

\medskip 
{Another approach to defining second-order generalized derivative constructions is based on second-order Taylor expansions, as developed in \cite[Chapter 13]{Rockafellar98}.} {We  recall the concept of {\em second subderivatives}, which provides a generalized second-order derivative framework for characterizing the convexity of nonsmooth lower semicontinuous functions.} The concept of second subderivatives was introduced by Rockafellar (see, e.g., \cite{roc88,roc90}) and later developed for calculus rules under weak constraint qualification conditions in \cite{mms}, as well as for Newton-type methods in \cite{kmp24cod,BorisEbrahim}.  To proceed, for $\ox \in \R^n, \ov\in \R^n,$ we  consider the family of second-order finite differences
\begin{equation*}
\Delta^2_\tau\varphi(\bar{x},\ov)(u):=\frac{\varphi(\bar{x}+\tau u)-\varphi(\bar{x})-\tau\langle \ov,u\rangle}{\frac{1}{2}\tau^2}
\end{equation*}
and define the {\em second subderivative} of $\varphi$ at $\ox$ for $\ov\in\R^n$ and $w\in\R^n$ by
\begin{equation}\label{2ndsubder}
d^2\varphi(\ox,\ov)(w):=\liminf_{\tau\downarrow 0\atop u\to w}\Delta^2_\tau\varphi(\ox,\ov)(u).
\end{equation}
 {By Proposition~13.5 in \cite{Rockafellar98}, the function $d^2 \varphi(\ox,\ov)$ is l.s.c. and positively homogeneous of degree $2$, i.e., 
$$
d^2 \varphi(\ox,\ov)(\lm w)=\lm^2d^2\varphi(\ox,\ov) \quad \text{for all }\; \lm >0, \; w\in\R^n,
$$
and $d^2\varphi(\ox,\ov)(0)<\infty.$}
The function $\ph$ is said to be {\em twice epi-differentiable} at $\ox$  {for $\ov$} if for every $w\in\R^n$ and every choice of $\tau_k\downarrow 0$ there exists a sequence $w^k\to w$ such that $\Delta^2_{\tau_k}\varphi(\bar{x},\ov)(w^k)\to d^2\varphi(\ox,\ov)(w)$ as $k \to \infty.$   {Note from \cite[Proposition~13.5]{Rockafellar98} that $d^2\varphi(\ox,\ov)$ is not necessarily proper. When $d^2\varphi(\ox,\ov)$ is proper and $\varphi$ is twice epi-differentiable at $\ox$ for $\ov$, we say that $\varphi$ is {\em properly twice epi-differentiable} at $\ox$ for $\ov$.}
Twice epi-differentiability has been recognized as an important concept of second-order variational analysis with numerous applications to optimization; see the aforementioned monograph by Rockafellar and Wets \cite{Rockafellar98}. 
We next introduce the notion of {\em generalized twice differentiability} of functions, originally proposed in \cite{roc}. This concept has been the subject of renewed investigation in recent works \cite{dsw25,quadcharvar,gtdquad}. Given a  {symmetric matrix} $A \in \R^{n \times n}$, denote the {\em quadratic form} associated with the matrix $A$ by
\begin{equation}\label{qua}
q_A(w) := \la w, Aw \ra\;\mbox{ for all }\; w \in \R^n.
\end{equation} 
A function $q: \R^n \to \overline{\R}$ is called a \textit{generalized quadratic form} if it is expressible as $q = q_A + \delta_L$, where $L$ is a linear subspace of $\R^n$, and where $A \in \R^{n \times n}$ is a symmetric matrix. It is said that $\varphi: \R^n \to \overline{\R}$ is \textit{generalized twice differentiable} at $\bar{x}$ for $\bar{v} \in \partial \varphi(\bar{x})$ if it is twice epi-differentiable at $\bar{x}$ for $\bar{v}$ and if the second-order subderivative $d^2 \varphi(\bar{x},\bar{v})$ is a generalized quadratic form. 

\medskip 
Let us now review some classes of extended-real-valued functions $\varphi\colon\R^n\to\oR$ broadly used in the paper. For $s \in \R$, the function $\varphi$ is called {\em  $s$-convex} on  {a convex set} $\Omega$   if 
\begin{equation}\label{ineqstrongfunction}
{\varphi((1-\lambda)x +\lambda y)\le(1-\lambda)\varphi(x)+\lambda \varphi(y) -\frac{s}{2}\lambda(1-\lambda)\|x-y\|^2,}
\end{equation}
for any $x,y\in\Omega,\;\lambda\in(0,1)$. 
Specifically, the function $\varphi$ is called convex when $s = 0$, {\em $s$-strongly convex} when $s > 0$, and {\em $(-s)$-weakly convex} when $s < 0$.  
An l.s.c.\ function $\varphi$ is called {\em prox-regular} at a point $\bar{x} \in \dom \varphi$ for a vector $\bar{v} \in \partial \varphi(\bar{x})$ if there exist constants $\epsilon > 0$ and $r \ge 0$ such that
\begin{equation}\label{prox}
    \varphi(x) \ge \varphi(u) + \langle v, x-u \rangle - \frac{r}{2} \|x-u\|^2
\end{equation}
holds for all $x \in \mathbb{B}_\epsilon(\bar{x})$ and $(u,v) \in \gph \partial \varphi \cap (\mathbb{B}_\epsilon(\bar{x}) \times \mathbb{B}_\epsilon(\bar{v}))$ with $\varphi(u) < \varphi(\bar{x}) + \epsilon$. If \eqref{prox} is satisfied for every $v \in \partial \varphi(\bar{x})$, we say that $\varphi$ is prox-regular at $\bar{x}$.  
The function $\varphi$ is said to be {\em subdifferentially continuous} at $\bar{x}$ for $\bar{v} \in \partial \varphi(\bar{x})$ if, for any $\epsilon > 0$, there exists $\delta > 0$ such that $|\varphi(x) - \varphi(\bar{x})| < \epsilon$ whenever $(x,v) \in \gph \partial \varphi \cap (\mathbb{B}_\delta(\bar{x}) \times \mathbb{B}_\delta(\bar{v}))$. When this condition holds for all $v \in \partial \varphi(\bar{x})$, we say that $\varphi$ is subdifferentially continuous at $\bar{x}$.  
An important observation is that if $\varphi$ is subdifferentially continuous at $\bar{x}$ for $\bar{v}$, the restriction ``$\varphi(u) < \varphi(\bar{x}) + \epsilon$" in the prox-regularity definition can be omitted.   
Functions that are both prox-regular and subdifferentially continuous are referred to as {\em continuously prox-regular}. This class encompasses many important functions in second-order variational analysis, including commonly used types in applications such as  convex functions, $\mathcal{C}^2$-smooth functions, amenable functions; see \cite[Chapter~13]{Rockafellar98}.  {Moreover, it follows from the subdifferential characterization of weakly convex functions in \cite[Theorem~3.1]{dm05} that any weakly convex function $\varphi$ automatically satisfies condition~\eqref{prox}, with subdifferential continuity also being automatically satisfied. Hence, $\varphi$ is continuously prox-regular.}

\medskip
{We next discuss the relationships among the second-order generalized derivative constructions introduced above, namely second subderivatives, second-order subdifferentials, and subgradient graphical derivatives. First, we have the following inclusions, which describe the relationship between the combined second-order subdifferential and the limiting second-order subdifferential:
\begin{equation}\label{relacombinedlimit}
\breve{\partial}^2\varphi(x,y)(w) \subset  {\partial}^2\varphi(x,y)(w) \quad \text{for all }(x,y)\in\gph\partial\varphi, \; w \in \R^n.
\end{equation}
In general, we do not have an inclusion relationship analogous to \eqref{relacombinedlimit} between second-order subdifferentials and subgradient graphical derivatives. However, when $\varphi$ is twice epi-differentiable and continuously prox-regular at $\ox$ for $\ov \in \partial\varphi(\ox)$, it follows from \cite[Theorem~1.1]{rz97} that we have the inclusion:
$$
(D\partial\varphi)(\ox,\ov)(w) \subset {\partial}^2\varphi(\ox,\ov)(w) \quad \text{for all }\; w \in \R^n.
$$
Due to the duality correspondence in \eqref{dua}, we obtain the following relationship between the combined second-order subdifferential and the subgradient graphical derivative:
$$
z \in \breve{\partial}^2\varphi(\ox,\ov)(w) \iff \langle z, u\rangle - \langle w,v\rangle \leq 0 \quad \text{for all }\; v \in (D\partial\varphi)(\ox,\ov)(u),\; u \in \R^n. 
$$
In the case where $\varphi$ is twice epi-differentiable, and continuously prox-regular at $\ox$ for $\ov \in \partial\varphi(\ox)$, it is known from \cite[Theorem 13.40]{Rockafellar98} that:
\begin{equation}\label{graphical2ndsub}
(D\partial\varphi)(\ox,\ov)  = \partial h \quad \text{with }\; h = \frac{1}{2}d^2\varphi(\ox,\ov).
\end{equation}}

\medskip 
The following two well-known constructions play a crucial role in this paper. Given 
an l.s.c.\ function $\varphi:\R^n\to\overline{\R}$ and a parameter value $\lambda>0$, the {\em Moreau envelope} $e_\lambda\varphi$ is  defined by
\begin{equation}\label{Moreau} 
	e_\lambda\varphi(x):=\inf\Big\{\varphi(y)+\frac{1}{2\lambda}\|y-x\|^2\;\Big|\;y\in \R^n\Big\},\quad x\in\R^n,
\end{equation}	
A function $\varphi$ is said to be {\em prox-bounded} if there exists $\lambda>0$ such that $e_\lambda\varphi(x)>-\infty$ for some $x\in\R^n$.  The supremum of the set of all such $\lambda$ is the threshold $\lambda_\varphi$ of the prox-boundedness for $\varphi$.  Note that it follows from \cite[Exercise 1.24 and Theorem 2.26]{Rockafellar98} that if $\varphi$ is $s$-convex,  then $\varphi$ is prox-bounded. \vspace*{0.03in}

Recall  that a function $\ph$ is of {\em class $\mathcal{C}^{1,1}$} around $\ox$ if it is ${\cal C}^1$-smooth and its gradient is Lipschitz continuous around this point. We now recall important properties of the Moreau envelope  that follows from  \cite[Lemma~2.5]{davisdrus22}, \cite[Proposition~13.37]{Rockafellar98} and \cite[Corollary 3.7]{wang10}. 

\begin{Proposition}[\bf Moreau envelopes  of prox-regular functions]\label{C11}  Let $\varphi\colon\R^n\to\oR$ be an l.s.c.\ and  prox-bounded function with threshold $\lambda_\varphi >0$. Suppose that $\varphi$ is continuously prox-regular at $\ox$ for $\ov\in\partial\varphi(\ox)$, then for all sufficiently small numbers $\lambda \in (0,\lambda_\varphi)$ there is a convex neighborhood $U_\lambda$ of $\ox+\lambda\ov$ on which the following hold:
\begin{itemize}
    \item[\bf (i)]  The Moreau envelope $e_\lambda\varphi$ from \eqref{Moreau} is of class ${\cal C}^{1,1}$ on the set $U_\lambda $.

 \item[\bf (ii)] The gradient of $e_\lm\ph$ is calculated by:
	\begin{equation}\label{GradEnvelope} 
		\nabla e_\lambda\varphi(x)=\big(\lambda I+(\partial\varphi)^{-1}\big)^{-1}(x)\;\mbox{ for all }\;x\in U_\lambda.
	\end{equation}
\end{itemize} 
In the case where $\varphi$ is weakly convex, then $\nabla e_\lambda\varphi$  is globally Lipschitz continuous on $\R^n$, and all the above conclusions hold with  $U_\lambda = \R^n$.
\end{Proposition}

\section{Characterizations  of Nonsmooth Convex Functions}\label{sec:graphical}
 In this section, we develop characterizations for l.s.c., proper, convex functions  via   the second-order generalized derivative constructions.

\medskip 
 {To proceed with our main results, we need some auxiliary results.} 
Let $\mathcal{L}^m$ be the $\sigma$-algebra of the Lebesgue measurable subsets on $\R^m$. We denote the Lebesgue measure on $\R^n$ and $\R$ by $\mu$ and $\lambda$, respectively. 
\begin{Lemma}\label{measure} Let $\Omega\subset \R^n$ be an open, convex set, and $\Omega_0 \subset \Omega$ such that $\Omega_0\in \mathcal{L}^n$ and $\mu(\Omega\setminus\Omega_0) =0$. Suppose that $a, b \in \Omega$ are distinct points.	Then there exist sequences of vectors $(a^k) \subset \Omega$, $(b^k)\subset \Omega$ such that $a^k \to a, b^k \to b$ and 
	\begin{equation} \label{LebesR}
	\lambda \left(\{t \in [0,1]: a^k + t(b^k-a^k)\in \Omega_0 \} \right) =1 \quad \text{for all }\; k \in \N.
	\end{equation} 
\end{Lemma}
\begin{proof} Putting $\Omega_1 : = (\R^n\setminus\Omega)\cup \Omega_0$, we have $\R^n\setminus\Omega_1 = \Omega\setminus\Omega_0$, which implies that 
$$
\mu (\R^n\setminus\Omega_1) = \mu (\Omega\setminus\Omega_0) =0.
$$
It follows from the openness of $\Omega$ that $\Omega \in \mathcal{L}^n$, and thus $\Omega_1 \in \mathcal{L}^n$. 	Using \cite[Lemma 4.1]{ChieuChuongYaoYen}, there exist sequences of vectors $a^k\to a$ and $b^k \to b$ such that
\begin{equation}\label{omega1}
\lambda \left(\{t \in [0,1]: a^k + t(b^k-a^k)\in \Omega_1 \} \right) =1, \quad \text{for all }\; k \in \N. 
\end{equation} 
Moreover, due to the openness of $\Omega$, $a^k, b^k \in \Omega$ for all sufficiently large $k \in \N$. Without loss of generality, we can assume that $(a^k) \subset \Omega$ and $(b^k)\subset \Omega$. By the convexity of $\Omega$, we have that
\begin{equation}\label{2setomega}
\{t \in [0,1]: a^k + t(b^k-a^k)\in \Omega_1 \} = \{t \in [0,1]: a^k + t(b^k-a^k)\in \Omega_0 \} \quad \forall k \in \N. 
\end{equation} 
Combining \eqref{omega1} and \eqref{2setomega}, we obtain \eqref{LebesR}. 
	
\end{proof}

 {The following auxiliary result provides a characterization of differentiable convex functions with Lipschitz continuous gradients on an open set.}

\begin{Theorem}[\bf characterizations of differentiable convex functions on an open set]\label{2ndLipconvex}  Let   $\varphi:\R^n\to\oR$ be differentiable on  a nonempty open convex set $\Omega \subset \R^n$. Suppose that $\nabla \varphi$ is Lipschitz continuous on $\Omega$. Then the following assertions are equivalent:
\begin{itemize}
    \item[\bf (i)] $\varphi$ is convex on $\Omega$.
    \item[\bf (ii)] $\nabla^2 \varphi(x)$ is positive semidefinite for all $x \in \Omega_0$, where
    $$
\Omega_0:= \{x \in \Omega|\; \varphi \; \text{is twice differentiable at } x \}.
$$
\end{itemize}

\end{Theorem}
\begin{proof}
{We first justify} the implication [{\bf (i)} $\Longrightarrow$ {\bf (ii)}].
To proceed, suppose that $\varphi$ is convex on $\Omega$, we show that $D\nabla\varphi(x)$ is positive semidefinite for all $x \in \Omega$ in the sense that
\begin{equation}\label{PSDgraphical}
\langle z, w \rangle \geq 0 \quad \text{for all }\; z \in D\nabla\varphi(x)(w),\; x \in \Omega, \; w \in \R^n.
\end{equation} 
Indeed, we pick any 
$x \in \Omega$, $w \in \R^n$ and $z \in D\nabla\varphi(x)(w)$. Then there exist sequences $t_k \downarrow 0$, $w_k \to w, z_k \to z$ such that
$$
(x,\nabla\varphi(x)) + t_k (w_k,z_k) \in \text{\rm gph}\nabla\varphi, \quad \text{for all }\; k \in \N, 
$$
which implies that $\nabla \varphi(x+t_k w_k) = \nabla\varphi(x) + t_k z_k$ for all $k \in \N$. Due to the first-order characterization of differentiable convex functions and  the convexity of $\varphi$ on $\Omega$, we have 
$$
\langle \nabla \varphi(x+t_k w_k) -\nabla\varphi(x), x +t_k w_k - x\rangle \geq 0, \quad \text{for all sufficiently large }\; k \in \N,
$$
or $t_k^2\langle z_k,  w_k\rangle \geq 0$, which implies that $\langle z_k, w_k\rangle \geq 0$ for sufficiently large $k \in \N$. Letting $k \to \infty$ in this inequality, we obtain the inequality $\langle z, w\rangle \geq 0$, which justifies \eqref{PSDgraphical}.  
Moreover, it is worth noting that
$$
D\nabla\varphi(x)=\{\nabla^2\varphi(x)\}
\quad \text{for all } x\in\Omega_0.
$$
Combining the latter with \eqref{PSDgraphical},  we obtain {\bf (ii)}. 

\color{black}
\medskip Next, suppose that {\bf (ii)} holds, we show that $\varphi$ is convex on $\Omega$. It follows from the Lipschitz continuity of $\nabla\varphi$ on $\Omega$ and the Rademacher  theorem that $\Omega_0\in \mathcal{L}^n$ and $\mu(\Omega\setminus\Omega_0)=0$. Let $x, y \in \Omega$ and $x \ne y$. 
By Lemma \ref{measure}, there exist sequences $(x^k)\subset \Omega$, $(y^k)\subset \Omega$ such that $x^k \to x$, $y^k \to y$ and 
$$
\lambda \left(\{t \in [0,1]: x^k + t(y^k-x^k)\in \Omega_0 \} \right) =1 \quad \text{for all }\; k \in \N.
$$
Consider the function $f: [0,1]\to \R$ given by $f(t):= \langle \nabla \varphi(x^k+t(y^k-x^k)),y^k-x^k\rangle$ for all $t \in [0,1]$. Applying the Newton-Leibniz formula to this function, we obtain 
\begin{align}\label{NL}
\langle \nabla\varphi(y^k)-\nabla\varphi(x^k),y^k-x^k\rangle & = f(1) -f(0) = \int_{0}^{1}f'(t)dt \nonumber\\
&=\int_{T^k} \langle \nabla^2 \varphi(x^k+t(y^k-x^k))(y^k-x^k),y^k-x^k\rangle dt
\end{align}
where $T^k:= \{t\in [0,1]|\; x^k + t(y^k-x^k)\in \Omega_0 \}$. Moreover, for any $t \in T^k$, we have $x^k + t(y^k-x^k)\in \Omega_0$, which implies that 
\begin{equation}\label{NL2}
\langle \nabla^2 \varphi(x^k+t(y^k-x^k))(y^k-x^k),y^k-x^k\rangle \geq 0 \quad \text{for all }\; t \in T^k.
\end{equation} 
Combining \eqref{NL} and \eqref{NL2}, we have $\langle \nabla\varphi(y^k)-\nabla\varphi(x^k),y^k-x^k\rangle \geq 0$ for all $k \in \N$. Letting $k \to \infty$, we obtain the inequality $\langle \nabla\varphi(y)-\nabla\varphi(x),y-x\rangle \geq 0$, which justifies the convexity of $\varphi$ on $\Omega$. The proof is complete. 
\end{proof}
{Another important auxiliary result provides a characterization of convex functions in terms of their Moreau envelopes; see \cite[Theorem 3.17]{wang10}.}

\begin{Lemma}[\bf characterization of convex functions via Moreau envelopes]\label{convexMoreau}  Suppose that  $\varphi:\R^n\to\overline{\R}$ is a proper, l.s.c.\ and prox-bounded function  with
 threshold $\lambda_\varphi>0$. Then $\varphi$ is convex if and only if  the Moreau envelope $e_\lambda \varphi$ is convex for some $\lambda\in(0,\lambda_\varphi)$.  
\end{Lemma}

{Now we are ready to present the following new characterizations of convex functions in terms of second subderivatives in \eqref{2ndsubder}.}
\begin{Theorem}[\bf convexity of functions via second subderivatives]\label{cvvia2ndsub} Let $\varphi:\R^n \to \oR$ be a proper, l.s.c and weakly convex function. Then the following assertions are equivalent: 
\begin{itemize}
    \item[\bf (i)] $\varphi$ is convex.
    \item[\bf (ii)] For all $w \in \R^n$, $(x,v) \in \gph \partial \varphi$, we have $d^2\varphi(x,v)(w) \geq 0$.
    \item[\bf (iii)] For all $w \in \R^n$,  $(x,v) \in \gph \partial \varphi$ such that $\varphi$ is twice epi-differentiable at $x$ for $v$, we have  $d^2\varphi(x,v)(w) \geq 0$.
    \item[\bf (iv)] For all $w \in \R^n$,   $(x,v) \in \gph \partial \varphi$ such that $\varphi$ is generalized twice differentiable at $x$ for $v$, we have  $d^2\varphi(x,v)(w) \geq 0$. 
\end{itemize}
\end{Theorem}
\begin{proof} The implication [{\bf (i)} $\Longrightarrow$ {\bf (ii)}] follows from \cite[Proposition 13.20]{Rockafellar98}. It is obvious that {\bf (ii)} $\Longrightarrow$ {\bf (iii)} $\Longrightarrow$ {\bf (iv)}, and thus it remains only to verify the implication [{\bf (iv)} $\Longrightarrow$ {\bf (i)}]. Indeed, using Proposition \ref{C11},  the weak convexity of $\varphi$ tells us that $\varphi$ is prox-bounded with threshold $\lambda_\varphi>0$,    the Moreau envelope $e_\lambda\varphi$ is continuously differentiable  with {Lipschitzian gradient}  on $\R^n$ for any sufficiently small  $\lambda \in (0,\lambda_\varphi)$, and \eqref{GradEnvelope} is satisfied with $U_\lambda =\R^n$.  Pick any such $\lambda $, we next show that $e_\lambda \varphi$ is convex. To proceed, fix any $u \in \R^n$ such that $e_\lambda\varphi$ is twice differentiable at $u$, we will verify that $\nabla^2 e_\lambda \varphi(u)$ is positive semidefinite. Putting 
$$
x:= u - \lambda \nabla e_\lambda \varphi(u) \quad \text{and }\; v:= \nabla e_\lambda \varphi(u),
$$
we deduce from \eqref{GradEnvelope} that $(x,v) \in \gph\partial\varphi$, and $u= x + \lambda v$. {It follows from \cite[Theorem~5.10]{gtdquad} that} $\varphi$ is generalized twice differentiable at $x$ for $v$.  Since {\bf (iv)} holds, we have $d^2\varphi(x,v) \geq 0$. {It follows from \cite[Proposition 5.5]{gtdquad} that}
$$
{e_{\lambda/2} \left[d^2 \varphi(x,v) \right](w) = d^2 e_\lambda \varphi \left(x+\lambda v,  v \right)(w) =  \langle \nabla^2 e_\lambda\varphi(u) w, w\rangle \quad \text{for all }\; w \in \R^n,}
$$
which implies that $\nabla^2 e_\lambda \varphi(u)$ is positive semidefinite since $d^2\varphi(x,v) \geq 0$.   Therefore, $e_\lambda\varphi$ is convex by Theorem \ref{2ndLipconvex}. Applying Lemma \ref{convexMoreau}, we conclude that $\varphi$ is convex, which completes the proof.   

\end{proof}

\begin{Remark}\rm 
It is worth emphasizing that the convexity characterization in
Theorem~\ref{cvvia2ndsub} concerns the nonnegativity of the second
subderivatives, rather than the convexity of the second-subderivative
functions themselves. We also wish to distinguish between two kinds
of second subderivatives: the second subderivative
$d^2\varphi(\bar{x},v)$, which depends on a specified subgradient
$v\in\partial\varphi(\bar{x})$, and the second subderivative
$d^2\varphi(\bar{x})$, which does not involve a specified
subgradient. The latter is closely related to the one-sided second
directional derivative $\varphi^{\prime\prime}(\bar{x},w)$. However, 
$d^2\varphi(\bar{x},v)$ and $\varphi^{\prime\prime}(\bar{x},w)$ do not necessarily coincide, as illustrated by the
example on \cite[Page 590]{Rockafellar98}.
\end{Remark}
\color{black}
Using the above characterizations of convex functions, we derive the corresponding results for strongly convex functions through second subderivatives.
\begin{Corollary}[\bf strong convexity via second  subderivatives]\label{co:strong} Let $\varphi:\R^n \to \oR$ be a proper, l.s.c and weakly convex function. {For $\kappa >0$, the following assertions are equivalent:}
\begin{itemize}
    \item[\bf (i)] {$\varphi$ is $\kappa$-strongly convex.}
    \item[\bf (ii)] For all $w \in \R^n$, $(x,v) \in \gph \partial \varphi$, we have $d^2\varphi(x,v)(w) \geq \kappa\|w\|^2$.
    \item[\bf (iii)] For all $w \in \R^n$,   $(x,v) \in \gph \partial \varphi$ such that $\varphi$ is twice epi-differentiable at $x$ for $v$, we have $d^2\varphi(x,v)(w) \geq \kappa\|w\|^2$.
    \item[\bf (iv)] For all $w \in \R^n$,   $(x,v) \in \gph \partial \varphi$ such that $\varphi$ is generalized twice differentiable at $x$ for $v$, we have $d^2\varphi(x,v)(w) \geq \kappa\|w\|^2$.
\end{itemize}
\end{Corollary} 
\begin{proof}  We first verify the implication [{\bf (i)} $\Longrightarrow$ {\bf (ii)}]. Indeed, suppose that  {$\varphi$ is $\kappa$-strongly convex}, and consider the function  {$\psi: \R^n\to \oR$} given by 
\begin{equation}\label{shiftpsi}
\psi(x)= \varphi(x) -\frac{\kappa}{2}\|x\|^2 \quad \text{for all } x \in \R^n.
\end{equation}
Using the sum rule for limiting subdifferentials taken from \cite[Proposition 1.30]{Mor18}, we have 
$$
\partial \psi(x) = \partial \varphi(x) - \kappa x
\quad \text{for all } x \in \R^n. 
$$ 
Using the sum rule for second subderivatives in \cite[Proposition 4.8]{gtdquad}, we obtain the equality
\begin{equation}\label{sumrule2ndsub}
d^2 \varphi(x,v)(w) = d^2 \psi(x,v-\kappa x)(w) + \kappa \|w\|^2 \quad \text{for all }\; (x,v) \in \gph\partial\varphi,\; w \in \R^n.
\end{equation}
Since  {$\varphi$ is $\kappa$-strongly convex,} it follows that $\psi$ is convex, and thus $d^2\psi(x,y)(w) \geq 0$ for all $(x,y) \in \gph \partial\psi$ and $w \in \R^n$ by using Theorem \ref{cvvia2ndsub}. Combining the latter with  \eqref{sumrule2ndsub}, we obtain {\bf (ii)}. 

\medskip 
It is obvious that {\bf (ii)} $\Longrightarrow$ {\bf (iii)} $\Longrightarrow$ {\bf (iv)}, and thus it remains only to verify the implication [{\bf (iv)} $\Longrightarrow$ {\bf (i)}]. Indeed, suppose that {\bf (iv)} holds. Consider $\psi$ given in \eqref{shiftpsi}, and we show that $\psi$ is convex by using Theorem \ref{cvvia2ndsub}. To proceed, fix any $w \in \R^n$, $(x,y) \in \gph\partial\psi$ such that $\psi$ is generalized twice differentiable at $x$ for $y$. 
Using \cite[Proposition 4.7]{gtdquad}, we deduce that $\varphi$ is generalized twice differentiable at $x$ for $y +\kappa x\in \partial\varphi(x)$. By  {\bf (iv)}, we deduce that $d^2\varphi(x,y+\kappa x)(w) \geq \kappa\|w\|^2$. Combining this with \eqref{sumrule2ndsub}, it follows that $d^2\psi(x,y)(w) \geq 0$, and thus $\psi$ is convex by  Theorem \ref{cvvia2ndsub}. This means that $\varphi$ is convex, which completes the proof. 
    
\end{proof}

 {Our next characterization of nonsmooth convex functions is based on the notion of the \textit{subgradient graphical derivative}. Our next results are direct consequences of Theorem~\ref{cvvia2ndsub} and Corollary~\ref{co:strong}. To proceed with our next  results, we first state the following lemma, which is an immediate consequence of \eqref{graphical2ndsub} and \cite[Lemma 3.6]{chnt}.}

\begin{Lemma}\label{d2andD} Let $\varphi:\R^n\to \oR$ be continuously prox-regular and properly twice epi-differentiable  at $\ox$ for $\ov \in \partial\varphi(\ox)$. Then for any $w \in \R^n$, $z \in (D\partial\varphi)(\ox,\ov)(w)$, we have 
$$
\langle z , w\rangle = d^2\varphi(\ox,\ov)(w). 
$$
\end{Lemma}

Now we are ready to derive another characterization of convex functions  via subgradient graphical derivatives under the additional weak convexity. 

\begin{Corollary}[\bf characterization of nonsmooth convex functions via subgradient graphical derivatives]\label{convexviagranonsmooth} Let $\varphi:\R^n \to \oR$ be a proper, l.s.c and weakly convex function. Then the following assertions are equivalent: 
\begin{itemize}
    \item[\bf (i)] $\varphi$ is convex.
    \item[\bf (ii)] For all $(x,v) \in \gph \partial \varphi$,  $(D\partial\varphi)(x,v)$ is positive semidefinite in the sense that
\begin{equation}\label{PSDgraphicalconvex}
\langle z, w\rangle \geq 0 \quad \text{for all } \; z \in (D\partial\varphi)(x,v)(w),\; w \in \R^n. 
\end{equation}
    \item[\bf (iii)] For all $(x,v) \in \gph \partial \varphi$ such that $\varphi$ is twice epi-differentiable at $x$ for $v$, we have  $(D\partial\varphi)(x,v)$ is positive semidefinite in the sense of \eqref{PSDgraphicalconvex}. 
    \item[\bf (iv)] For all $(x,v) \in \gph \partial \varphi$ such that $\varphi$ is generalized twice differentiable at $x$ for $v$, we have  $(D\partial\varphi)(x,v)$ is positive semidefinite in the sense of \eqref{PSDgraphicalconvex}.
\end{itemize}
\end{Corollary}
\begin{proof} We first verify the implication [{\bf (i)} $\Longrightarrow$ {\bf (ii)}]. Pick any $z \in (D\partial\varphi)(x,v)(w)$, where $(x,v)\in \gph \partial\varphi$ and $w \in \R^n$. Then we can find sequences $t_k\downarrow 0$, $w_k \to w$    as $k \to \infty$  such that 
\begin{equation}\label{ingraphdvp}
v + t_k z_k \in  \partial \varphi(x + t_k w_k) \quad \text{for all } k \in \N. 
\end{equation}
It follows from the convexity of $\varphi$ that $\partial\varphi$ is monotone. Combining this with \eqref{ingraphdvp} and the fact that $(x,v)\in\gph \partial\varphi$, we deduce that 
$$
\langle x +t_k w_k -x, v + t_k z_k - v\rangle \geq 0 \quad \text{for all }k \in \N,
$$
which implies that $\langle w_k, z_k\rangle \geq 0$ for all $k \in \N.$ Letting $k\to \infty$ in this inequality, we have $\langle w, z\rangle \geq 0$.  It is obvious that {\bf (ii)} $\Longrightarrow$ {\bf (iii)} $\Longrightarrow$ {\bf (iv)}, and thus it remains only to verify the implication [{\bf (iv)} $\Longrightarrow$ {\bf (i)}]. Indeed, suppose {\bf (iv)} holds, we show that for any $w \in \R^n$, $(x,v) \in \gph\partial\varphi$ such that $\varphi$ is generalized twice differentiable at $x$ for $v$, the second subderivative $d^2\varphi(x,v)(w) \geq 0$. Clearly, this follows immediately from  Lemma \ref{d2andD} and  the assertion {\bf (iv)}. Hence, by using Theorem \ref{cvvia2ndsub}, we obtain the assertion {\bf (i)}, which completes the proof. 
    
\end{proof}

Building on the above characterizations of convex functions, we now present the corresponding consequences for strongly convex functions.
\begin{Corollary}[\bf strong  convexity via subgradient graphical derivatives]   Let $\varphi:\R^n \to \oR$ be a proper, l.s.c and weakly convex function. For $\kappa >0$, the following assertions are equivalent: 
\begin{itemize}
    \item[\bf (i)] $\varphi$ is $\kappa$-strongly convex.
    \item[\bf (ii)] For all $(x,v) \in \gph \partial \varphi$,  $(D\partial\varphi)(x,v)$ is $\kappa$-positive definite in the sense that
\begin{equation}\label{PSDgraphicalstrconvex}
\langle z, w\rangle \geq \kappa\|w\|^2 \quad \text{for all } \; z \in (D\partial\varphi)(x,v)(w),\; w \in \R^n. 
\end{equation}
    \item[\bf (iii)] For all $(x,v) \in \gph \partial \varphi$ such that $\varphi$ is twice epi-differentiable at $x$ for $v$, we have  $(D\partial\varphi)(x,v)$ is $\kappa$-positive definite in the sense of \eqref{PSDgraphicalstrconvex}. 
    \item[\bf (iv)] For all $(x,v) \in \gph \partial \varphi$ such that $\varphi$ is generalized twice differentiable at $x$ for $v$, we have  $(D\partial\varphi)(x,v)$ is $\kappa$-positive definite in the sense of \eqref{PSDgraphicalstrconvex}.
\end{itemize}
\end{Corollary}
\begin{proof} The proof follows the same idea as that of Corollary~\ref{convexviagranonsmooth}, but uses Corollary~\ref{co:strong} in place of Theorem~\ref{cvvia2ndsub}, together with Lemma~\ref{d2andD}; hence, we omit the details.

\end{proof}
\color{black}
Next, we employ {\em second-order subdifferentials} \eqref{seccombine} to provide further characterizations of convexity for extended-real-valued functions.

\begin{Theorem}[\bf characterization of nonsmooth convex functions via second-order subdifferentials]\label{2ndcoderivativeconvex} Let $\varphi:\R^n \to \oR$ be a proper, l.s.c and weakly convex function. Then the following assertions are equivalent: 
\begin{itemize}
    \item[\bf (i)] $\varphi$ is convex.
    \item[\bf (ii)] For all $(x,v) \in \gph \partial \varphi$,  $\breve{\partial}^2\varphi(x,v)$ is positive semidefinite in the sense that
\begin{equation}\label{PSDcoderconvex}
\langle z, w\rangle \geq 0 \quad \text{for all } \; z \in \breve{\partial}^2\varphi(x,v)(w),\; w \in \R^n. 
\end{equation}
    \item[\bf (iii)] For all $(x,v) \in \gph \partial \varphi$ such that $\varphi$ is twice epi-differentiable at $x$ for $v$, we have  $\breve{\partial}^2\varphi(x,v)$ is positive semidefinite in the sense of \eqref{PSDcoderconvex}. 
    \item[\bf (iv)] For all $(x,v) \in \gph \partial \varphi$ such that $\varphi$ is generalized twice differentiable at $x$ for $v$, we have  $\breve{\partial}^2\varphi(x,v)$ is positive semidefinite in the sense of \eqref{PSDcoderconvex}.
    \item[\bf (v)] For all $(x,v) \in \gph \partial \varphi$,  ${\partial}^2\varphi(x,v)$ is positive semidefinite in the sense that
\begin{equation}\label{PSDlimitcoderconvex}
\langle z, w\rangle \geq 0 \quad \text{for all } \; z \in {\partial}^2\varphi(x,v)(w),\; w \in \R^n. 
\end{equation}
    \item[\bf (vi)] For all $(x,v) \in \gph \partial \varphi$ such that $\varphi$ is twice epi-differentiable at $x$ for $v$, we have  ${\partial}^2\varphi(x,v)$ is positive semidefinite in the sense of \eqref{PSDlimitcoderconvex}. 
    \item[\bf (vii)] For all $(x,v) \in \gph \partial \varphi$ such that $\varphi$ is generalized twice differentiable at $x$ for $v$, we have  ${\partial}^2\varphi(x,v)$ is positive semidefinite in the sense of \eqref{PSDlimitcoderconvex}.
\end{itemize}
\end{Theorem}
\begin{proof} The roadmap for proving this result is to establish the following chains of implications:
$$
\text{[{\bf (i)}$\Longrightarrow${\bf(ii)}$\Longrightarrow${\bf (iii)}$\Longrightarrow${\bf (iv)}$\Longrightarrow${\bf(i)}]} \quad \text{and }\; \text{[{\bf (i)}$\Longrightarrow${\bf(v)}$\Longrightarrow${\bf (vi)}$\Longrightarrow${\bf (vii)}$\Longrightarrow${\bf(i)}]}.
$$
The implication [{\bf (i)}$\Longrightarrow${\bf (v)}] follows immediately from  \cite[Theorem 3.2]{ChieuChuongYaoYen}. Due to \eqref{relacombinedlimit}, we obtain  [{\bf (v)}$\Longrightarrow${\bf (ii)}], which  tells us that [{\bf (i)}$\Longrightarrow${\bf (ii)}]. The chains of implications [{\bf(ii)}$\Longrightarrow${\bf (iii)}$\Longrightarrow${\bf (iv)}] and [{\bf(v)}$\Longrightarrow${\bf (vi)}$\Longrightarrow${\bf (vii)}] are obvious. Using \eqref{relacombinedlimit} again, we also obtain the implication [{\bf (vii)}$\Longrightarrow${\bf (iv)}]. Hence, it remains to verify the implication [{\bf (iv)}$\Longrightarrow${\bf (i)}]. Using Proposition \ref{C11},  the weak convexity of $\varphi$ tells us that $\varphi$ is prox-bounded with threshold $\lambda_\varphi>0$,    the Moreau envelope $e_\lambda\varphi$ is continuously differentiable  with Lipschitzian derivative  on $\R^n$ for any sufficiently small  $\lambda \in (0,\lambda_\varphi)$, and \eqref{GradEnvelope} is satisfied with $U_\lambda =\R^n$.  Pick any such $\lambda $, we next show that $e_\lambda \varphi$ is convex. To proceed, we will verify the convexity of  $e_\lambda \varphi$   on $\R^n$ by Theorem \ref{2ndLipconvex}. Indeed, fix any $u \in \R^n$ such that $e_\lambda\varphi$ is twice differentiable at $u$, we will verify that $\nabla^2 e_\lambda \varphi(u)$ is positive semidefinite. Putting 
$$
x:= u - \lambda \nabla e_\lambda \varphi(u) \quad \text{and }\; v:= \nabla e_\lambda \varphi(u),
$$
we deduce from \eqref{GradEnvelope} that $(x,v) \in \gph\partial\varphi$, and $u= x + \lambda v$. It follows from \cite[Theorem~5.10]{gtdquad} that $\varphi$ is generalized twice differentiable at $x$ for $v$.  Since {\bf (iv)} holds, we have $\breve{\partial}^2\varphi(x,v)$ is positive semidefinite in the sense of \eqref{PSDcoderconvex}. 
Since $e_\lm \varphi$ is twice differentiable at $u$, we have  $ \breve{\partial}^2 e_\lambda \varphi(u)(w) =\{\nabla^2 e_\lm \varphi(u)w\}$ for all $w \in \R^n$. For any $w \in \R^n$, we put $z:= \nabla^2 e_\lm \varphi(u)w$ and deduce that
$$
-w \in \widehat{D}^*(\nabla e_\lambda\varphi)^{-1}(\nabla e_\lambda \varphi(u),u)(-z).
$$
Moreover,  by using  the   sum rule from \cite[Theorem 1.62]{Mordukhovich06}, we get
\begin{eqnarray*}
\widehat{D}^*(\nabla e_\lambda\varphi)^{-1}(\nabla e_\lambda \varphi(u),u)(-z)&=&\widehat{D}^*(\lambda I + (\partial\varphi)^{-1})(\nabla e_\lambda \varphi(u),u)(-z)\\
	& =& -\lambda z + \widehat{D}^*(\partial\varphi)^{-1}(\nabla e_\lambda \varphi(u),u-\lambda \nabla e_\lambda \varphi(u))(-z).
\end{eqnarray*}
It follows that
$-w+ \lambda z  \in \widehat{D}^*(\partial\varphi)^{-1}(\nabla e_\lambda \varphi(u),u-\lambda \nabla e_\lambda \varphi(u))(-z).$
In other words,
$$
z \in (\widehat{D}^*\partial\varphi)(u-\lambda \nabla e_\lambda \varphi(u), \nabla e_\lambda \varphi(u))(w-\lambda z) = \breve{\partial}^2\varphi(x,v)(w-\lm z). 
$$
Due to the positive semidefiniteness of $\breve{\partial}^2\varphi(x,v)$,  we obtain the inequality
$\langle z, w-\lambda z\rangle\geq 0$ and so $\langle z,u\rangle\geq 0$, or $\langle \nabla^2 e_\lm \varphi(u)w,w\rangle \geq 0$.  Hence, $e_\lambda\varphi$ is convex. Using Lemma \ref{convexMoreau}, we deduce that $\varphi$ is convex, which completes the proof of this theorem. 
    
\end{proof}
\color{black}
The preceding characterizations of convex functions allow us to formulate equivalent statements for strongly convex functions in terms of second-order subdifferentials.

\begin{Corollary}[\bf strong convexity via second-order subdifferentials]\label{co:strongconvex} Let $\varphi:\R^n \to \oR$ be a proper, l.s.c and weakly convex function.  For $\kappa >0$, the following assertions are equivalent:
\begin{itemize}
    \item[\bf (i)] $\varphi$ is $\kappa$-strongly convex.
    \item[\bf (ii)] For all $(x,v) \in \gph \partial \varphi$,  $\breve{\partial}^2\varphi(x,v)$ is $\kappa$-positive definite in the sense that
\begin{equation}\label{PSDcoderconvex1}
\langle z, w\rangle \geq \kappa\|w\|^2 \quad \text{for all } \; z \in \breve{\partial}^2\varphi(x,v)(w),\; w \in \R^n. 
\end{equation}
    \item[\bf (iii)] For all $(x,v) \in \gph \partial \varphi$ such that $\varphi$ is twice epi-differentiable at $x$ for $v$, we have  $\breve{\partial}^2\varphi(x,v)$ is $\kappa$-positive definite in the sense of \eqref{PSDcoderconvex1}. 
    \item[\bf (iv)] For all $(x,v) \in \gph \partial \varphi$ such that $\varphi$ is generalized twice differentiable at $x$ for $v$, we have  $\breve{\partial}^2\varphi(x,v)$ is $\kappa$-positive definite in the sense of \eqref{PSDcoderconvex1}.
    \item[\bf (v)] For all $(x,v) \in \gph \partial \varphi$,  ${\partial}^2\varphi(x,v)$ is $\kappa$-positive definite in the sense that
\begin{equation}\label{PSDlimitcoderconvex1}
\langle z, w\rangle \geq \kappa\|w\|^2 \quad \text{for all } \; z \in {\partial}^2\varphi(x,v)(w),\; w \in \R^n. 
\end{equation}
    \item[\bf (vi)] For all $(x,v) \in \gph \partial \varphi$ such that $\varphi$ is twice epi-differentiable at $x$ for $v$, we have  ${\partial}^2\varphi(x,v)$ is $\kappa$-positive definite in the sense of \eqref{PSDlimitcoderconvex1}. 
    \item[\bf (vii)] For all $(x,v) \in \gph \partial \varphi$ such that $\varphi$ is generalized twice differentiable at $x$ for $v$, we have  ${\partial}^2\varphi(x,v)$ is $\kappa$-positive  definite in the sense of \eqref{PSDlimitcoderconvex1}.
\end{itemize}
\end{Corollary}
\begin{proof} The roadmap for proving this result is to establish the same chains of implications as in Theorem~\ref{2ndcoderivativeconvex}. The argument for the implication [{\bf (iv)}$\Longrightarrow${\bf (i)}] is very similar to that in the proof of Corollary~\ref{co:strong}. The main difference is that the sum rule for second subderivatives is replaced by the corresponding sum rule for second-order subdifferentials, with the remaining arguments involving second subderivatives replaced by their second-order subdifferential counterparts. Therefore, we omit the details. 
    
\end{proof}

\begin{Remark}[\bf discussion  on   Theorem~\ref{2ndcoderivativeconvex} and Corollary~\ref{co:strongconvex}] \rm  The proof of the equivalence [{\bf (i)}$\iff${\bf (ii)}$\iff${\bf (v)}] in both Theorem~\ref{2ndcoderivativeconvex} and Corollary~\ref{co:strongconvex} can also be obtained via a different approach. Indeed, it is well known in classical convex analysis that a proper l.s.c.\ function $\varphi$ is convex if and only if its subgradient mapping $\partial\varphi$ is maximally monotone. Consequently, weak convexity of $\varphi$ is equivalent to the hypomonotonicity of $\partial\varphi$ on its convex domain; see \cite[Chapter 12]{Rockafellar98}. Therefore, the aforementioned characterizations can be viewed as subgradient-mapping counterparts of the results established in \cite[Theorems 3.2 and 3.5]{clmn16} and \cite[Corollary~3.10]{clmn16}.

\medskip Our approach and proof are substantially different. They rely heavily on the relationship between the convexity of a function and that of its Moreau envelope established in Lemma~\ref{convexMoreau}, and further exploit the celebrated notion of weak convexity for extended-real-valued functions together with its rich properties and its close connections with the associated Moreau--Yosida regularization. In contrast to \cite{clmn16}, our proof proceeds directly without invoking characterizations of monotone operators and does not depend strongly on the specific structure of coderivatives. Furthermore, the methodology developed here applies uniformly to all generalized second-order derivatives considered in this paper. To the best of our knowledge, the other implications in Theorem~\ref{2ndcoderivativeconvex} and Corollary~\ref{co:strongconvex} have not been previously investigated in the literature. 
\end{Remark}

\begin{Remark}[\bf lower-$\mathcal C^2$ functions and convexity characterizations] \rm 
Recall that every lower-$\mathcal C^2$ function is locally
weakly convex; see \cite[Definition 10.29]{Rockafellar98}.
However, the convexity characterizations established in
this section require global weak convexity and therefore
cannot be directly applied to the entire class of
lower-$\mathcal C^2$ functions. Indeed, a lower-$\mathcal C^2$
function need not be globally weakly convex. For example,
$\varphi(x)=-x^4$ is lower-$\mathcal C^2$ but is not
globally weakly convex. Thus, extending the present
characterizations to the broader classes of locally weakly
convex and lower-$\mathcal C^2$ functions remains an
interesting direction for further investigation. 
\end{Remark}

\color{black}
\section{Discussions and Examples}\label{sec:ex}
In this section, we present several illustrative examples and discussions that complement the theoretical results established in the previous sections. In particular, these examples demonstrate the necessity of the additional weak convexity assumption in both the finite-valued and extended-real-valued settings, while also illustrating how the proposed characterizations can be verified and interpreted in concrete situations. To this end, we consider a variety of examples, including functions that frequently arise in optimization, variational analysis, and related practical models.

\medskip 
The first example concerns a one-dimensional finite-valued function that can be viewed as a special case of the $\ell_0$-norm, which is widely used in compressive sensing and sparse optimization. The variational analysis of the $\ell_0$-norm, including the computation of its second-order subdifferentials and the development of numerical Newton-type methods, has been extensively studied in recent works; see, e.g., \cite{kmp22convex,kmp24cod}. In what follows, this simple example demonstrates that, in the absence of weak convexity, the nonnegativity of the second subderivative in Theorem~\ref{cvvia2ndsub}, the positive semidefiniteness of the subgradient graphical derivative in Corollary~\ref{convexviagranonsmooth} and the positive semidefiniteness of second-order subdifferentials in~Theorem~\ref{2ndcoderivativeconvex} are not sufficient to guarantee the convexity of the function.
\begin{Example}[\bf $\ell_0$-norm]\rm Consider the function $\varphi:\R\to\R$ given by
$$
\varphi(x) = \begin{cases}
1 & \text{if }\; x \ne 0,\\
0 & \text{otherwise}. 
\end{cases}
$$
By direct computations, it is not hard to get 
$$
\partial \varphi(x)= \begin{cases}
\{0\} &\text{if }\; x \ne 0,\\
\R &\text{if }\; x =0,
\end{cases}
$$
$$
T_{{\rm \text{\rm gph}}\,\partial \varphi}(x, y) =  \begin{cases}
\R \times \{0\} &\text{if }\; x \ne  0, y = 0,\\
(\{0\}\times \R)\cup (\R \times \{0\}) &\text{if }\; x=0,y=0,\\
\{0\}\times \R &\text{if }\; x= 0, y\ne  0, 
\end{cases}
$$
$$
N_{{\rm \text{\rm gph}}\,\partial \varphi}(x, y) =  \begin{cases}
\{0\} \times \R &\text{if }\; x \ne 0, y =0,\\
(\{0\}\times \R)\cup (\R \times \{0\}) &\text{if }\; x=0,y=0,\\
\R \times \{0\} &\text{if }\; x=0, y\ne 0, 
\end{cases}
$$
which clearly imply that for each $(x,y) \in \gph \partial\varphi$, we have 
\begin{equation*} 
 (D\partial\varphi)(x,y)(v) = \partial^2\varphi(x,y)(v) = \begin{cases}
\{0\} & \text{if }\; x \ne 0, y =0, v \in \R\\
\{0\} & \text{if } \; x =0, y =0,  v \ne 0,\\
\R & \text{if } \; x =0, y =0,  v = 0,\\
\R & \text{if }\; x=0, y \ne 0, v=0,\\
\emptyset &\text{otherwise}.
\end{cases}    
\end{equation*}
Therefore, the subgradient graphical derivative $(D\partial\varphi)(x,y)$ and the second-order subdifferential $\partial^2\varphi(x,y)$ are positive semidefinite in the sense of    \eqref{PSDgraphicalconvex} and \eqref{PSDlimitcoderconvex}, respectively. Moreover, by the definition of second subderivatives in \eqref{2ndsubder}, we deduce that 
$$
d^2 \varphi(x,y)(w) = \begin{cases}
0 & \text{if }\; x\ne 0, y =0, w \in \R,\\
0 & \text{if }\; x= 0, y \in \R, w =0,\\
\infty & \text{if }\; x =0, y \in \R, w \ne 0,
\end{cases}
$$
and thus $d^2 \varphi(x,y)$ is nonnegative for all $(x,y) \in \gph\partial\varphi$. 
However, $\varphi$ is not convex, indicating that the weak convexity assumption in Section \ref{sec:graphical} cannot be removed. 
\end{Example}

We next present an example illustrating how our characterizations can be employed to verify the convexity of functions in situations where several well-known criteria based on generalized derivatives are not applicable. In particular, the convexity characterizations developed in \cite{ChieuChuongYaoYen,chieuyao11,chieuhuy,clmn16,nz23} are established for finite-valued, locally Lipschitz functions. By contrast, the following example demonstrates that our results can be applied to a broader class of functions that fall outside the scope of those existing characterizations.

\begin{Example}[\bf extended-real-valued composite functions]\label{ex:portfolio} \rm
We consider the function $\varphi:\R^n\to \oR$ defined by
\begin{equation}\label{composite}
\varphi(x)=g(x)+h(c(x)) \quad \text{for all }\; x \in \R^n, 
\end{equation}
where $g:\R^n\to \oR$ is a proper, l.s.c., convex function, $h:\R^m\to \R$ is a globally Lipschitz continuous convex function, and $c:\R^n\to \R^m$ is a $\mathcal{C}^2$-smooth mapping whose gradient $\nabla c$ is globally Lipschitz continuous. In general, the resulting function $\varphi$ is nonconvex.

\medskip 
By \cite[Lemma~4.2]{dp19}, the composition $h\circ c$ is weakly convex. Since the sum of a convex function and a weakly convex function is weakly convex, it follows that $\varphi$ is weakly convex. Therefore, the characterizations developed in Section \ref{sec:graphical} are applicable to this class of functions. 
Such models arise naturally in a variety of applications. For instance, in computational finance, portfolio optimization problems involving nonlinear market simulators often lead to highly nonconvex objective functions. In particular, the robust portfolio optimization model for banking foundations studied in \cite{agl23} belongs to the above class of functions. Another explicit example is given by the support function of a nonempty, bounded, closed, and convex set $K\subset\R^m$. Recall that the support function of $K$ is defined by
\[
h(u):=\sigma_K(u)
=\sup_{y\in K}\langle y,u\rangle.
\]
Then $h$ is convex and globally Lipschitz continuous.
Consequently, if $c:\R^n\to\R^m$ is continuously differentiable with globally Lipschitz continuous gradient, then the composite function $\sigma_K\circ c$ is weakly convex.
Moreover, \cite{dp19} discusses several additional practical models that fit within this framework; see, e.g.,  \cite[Examples 3.1--3.5]{dp19}.
\end{Example}
To further characterize the convexity of the function considered in Example \ref{ex:portfolio}, we provide the following result based on the preceding results established in Section \ref{sec:graphical}. 
\begin{Proposition}[\bf second-order sufficient conditions for convexity of extended-real-valued composite functions]\label{prop:2ndcompo} Let  $\varphi:\R^n 	\to  \oR$ be defined as in \eqref{composite}.
Suppose that for each $x \in \R^n$, $v  \in \partial g(x)$, $y \in \partial h(c(x))$,  the following condition holds:
\begin{equation}\label{PSDcomposite}
 d^2 g(x,v)(w)+ \langle \nabla^2\langle y, c\rangle (x)w,w\rangle + d^2 h(c(x),y)(\nabla c(x)w)   \geq 0 \quad \text{for all } w \in \R^n.
\end{equation}
Then $\varphi$ is a convex function.
\end{Proposition}
\begin{proof} Picking any $x\in\R^n$ and $y\in\partial h(c(x))$, and applying the chain rule for second subderivatives from \cite[Theorem 13.14]{Rockafellar98} (whose qualification condition is automatically satisfied since $h$ is convex and finite-valued), we obtain 
\begin{equation}\label{1stchainproof}
\partial (h\circ c)(x) = \nabla c(x)^* \partial h(c(x)),
\end{equation}
and
\begin{equation}\label{2ndchainproof}
d^2 (h\circ c)(x,\nabla c(x)^*y)(w)
\ge
\langle \nabla^2\langle y,c\rangle(x)w,w\rangle
+
d^2 h(c(x),y)(\nabla c(x)w)
\end{equation}
for all $w\in\R^n$. 
Now, let $v\in\partial g(x)$. Since $h\circ c$ is locally Lipschitz continuous, the qualification condition in the sum rule for second subderivatives from \cite[Proposition 13.19]{Rockafellar98} is automatically satisfied. Therefore,
\begin{equation}\label{1stsumproof}
\partial \varphi(x) = \partial g(x) + \partial (h\circ c)(x),
\end{equation}
and
\begin{equation}\label{2ndsumproof}
d^2\varphi(x,v+\nabla c(x)^*y)(w)
\ge
d^2 g(x,v)(w)
+
d^2(h\circ c)(x,\nabla c(x)^*y)(w).
\end{equation}
Combining \eqref{2ndchainproof}, \eqref{2ndsumproof} and the assumed condition \eqref{PSDcomposite} yields
$$
d^2\varphi(x,v+\nabla c(x)^*y)(w)\ge d^2 g(x,v)(w)+ \langle \nabla^2\langle y, c\rangle (x)w,w\rangle + d^2 h(c(x),y)(\nabla c(x)w) \geq 0
$$
for all $w\in\R^n$. Combining the latter with \eqref{1stchainproof} and \eqref{1stsumproof}, we deduce that $d^2\varphi(x,u)(w) \geq 0$ for all $(x,u)\in \gph\partial\varphi$ and $w\in \R^n$. Moreover, since $\varphi$ is weakly convex due to Example \ref{ex:portfolio}, it follows from Theorem~\ref{cvvia2ndsub} that $\varphi$ must be convex, which completes the proof. 

\end{proof}

To illustrate the above characterization, we consider the specific example  which belongs to the class of functions introduced in Example~\ref{ex:portfolio}.
\begin{Example} \rm Consider the function $\varphi:\R^n\to\R$ given as in \eqref{composite} where $n=2$, $m=1$, 
$$
g(x) =   \begin{cases}
 0 & \text{if }\; x_1 \leq  0, x_2 =0\\
 \infty & \text{otherwise}
\end{cases} \quad \text{for all }\; x=(x_1,x_2)\in\R^2,
$$
$$
h(u)= u \quad \text{for all } \; u\in\R, \quad \text{and }\; c(x)=x_1^2 -x_2^2 \quad  \text{for all }\; x=(x_1,x_2)\in\R^2.
$$
Note that $g$ can be represented as $g=\delta_C$, where $C:=\{x=(x_1,x_2)\in\R^2|\; x_1 \leq 0, x_2 =0 \} $.   More
precisely,
$$
\partial g(x)=
\begin{cases}
\{(0,\lambda)\mid \lambda\in\R\},
& \text{if }x_1<0,\ x_2=0,\\[1mm]
\{(\alpha,\lambda)\mid
\alpha\geq 0,\ \lambda\in\R\},
& \text{if }x_1=0,\ x_2=0,\\[1mm]
\emptyset,&\text{if }x\notin C.
\end{cases} \quad \text{for all }\; x =(x_1,x_2)\in \R^2.
$$
 Moreover, we have 
$$
\nabla h(u) = 1 \quad \text{for all }u \in \R, \quad   \nabla c(x) = (2x_1, -2x_2) \quad \text{for all }\; x=(x_1,x_2)\in\R^2,
$$
$$
\nabla^2  c(x)  =
\begin{pmatrix}
2&0\\
0&-2
\end{pmatrix} \quad \text{for all }\; x=(x_1,x_2)\in\R^2.
$$
It is clear that $h$  and $\nabla c$ are Lipschitz continuous. 
For any $(x,v)\in \gph \partial g$,  $y = \nabla h(c(x))= 1$, we have $x_1 \leq 0$,   which implies that  
\begin{equation}\label{cHess}
\langle \nabla^2 (y.c)(x)w,w\rangle  = 2w_1^2 - 2w_2^2 \quad \text{for all }\; w=(w_1,w_2) \in \R^2.
\end{equation}
Since $h(u)=u$, we have $\partial h(u)=\{1\}$ for all $u \in \R$ and
$$
d^2h(c(x),1)(\nabla c(x)w)=0
\quad\text{for all }x,w\in\R^2.
$$
Moreover, since $g=\delta_C$ is convex, we deduce from Theorem~\ref{cvvia2ndsub} that 
$d^2g(x,v)(w)\geq 0$ for every
$x\in C$, $v\in\partial g(x)$, and $w\in\R^2$. Fixing any $x=(x_1,x_2) \in C, v \in \partial g(x)$ and $w=(w_1,w_2) \in \R^2$, we consider the two following cases:

\medskip\noindent
{\bf Case 1.} $d^2g(x,v)(w)=+\infty$. In this case, \eqref{PSDcomposite} is automatically satisfied. 

\medskip\noindent
{\bf Case 2.} $0\leq d^2g(x,v)(w)<+\infty$. By the definition of the second subderivative, there exist
sequences $t_k\downarrow0$ and $w^k:=(w_1^k,w_2^k)\to w=(w_1,w_2)$ such that
$x+t_kw^k\in C$ for all sufficiently large $k$.
Since  every point in $C$ has its second
coordinate equal to zero, we obtain
$x_2+t_kw_2^k=0$, and hence $w_2^k=0$.
Passing to the limit gives $w_2=0$.
Combining this with \eqref{cHess}, we obtain
$$
\begin{aligned}
d^2g(x,v)(w)
+\langle \nabla^2(y.c)(x)w,w\rangle
+d^2h(c(x),1)(\nabla c(x)w)
=d^2g(x,v)(w)+2w_1^2\geq0.
\end{aligned}
$$
Thus, \eqref{PSDcomposite} holds.  In both cases, using Proposition~\ref{prop:2ndcompo}, we deduce that $\varphi$ is convex.  
\end{Example}

\color{black}
\section{Conclusion}\label{sec:conclusion} 
In this paper, we have developed new characterizations of convexity for nonsmooth {extended-real-valued} functions. By employing three  {second-order} generalized derivatives, namely graphical derivatives, second subderivatives, and  {second-order} subdifferentials, {we established characterizations for convexity in the nonsmooth setting}. These results not only enhance the theoretical understanding of convexity in nonsmooth optimization, but also provide a foundation for the design and convergence analysis of numerical algorithms in optimization and variational analysis. Future research will explore further applications of these characterizations in algorithm development and extend the framework to broader classes of nonsmooth functions.

\section*{Acknowledgements} \vspace*{-0.05in}
The author thanks Pham Duy Khanh for his valuable discussions. The author  is very grateful to two anonymous referees for their helpful remarks and suggestions, which allowed us to significantly improve the original presentation. 

\section*{Data Availability}  
This manuscript has no associated data.

\section*{Funding and Conflicts of Interests} 

This research was partly supported by the Early Career Scholars
Program 2026, University of North Dakota under project 43700-2375-UND0031286. The authors declare that the presented results are new, and there are no  conflicts of interest.
\end{document}